\documentclass[12pt]{amsart}
\usepackage[all]{xy}
\usepackage{amssymb}
\usepackage{amsthm}
\usepackage[pagebackref]{hyperref}
\hypersetup{colorlinks=true,allcolors=blue}
\usepackage{amsmath}
\usepackage{amscd,enumitem}
\usepackage{verbatim}
\usepackage{eurosym}
\usepackage{graphicx}
\usepackage{dsfont}

\usepackage{float}
\usepackage{color}
\usepackage{longtable}
\usepackage{dcolumn}
\usepackage[mathscr]{eucal}
\usepackage[all]{xy}
\usepackage{hyperref}
\usepackage[title]{appendix}
\usepackage[usenames,dvipsnames]{xcolor}
\usepackage{bbm}
\usepackage[textheight=8.7in, textwidth=6.75in]{geometry}
\usepackage{multirow}
\usepackage{caption}
\newtheorem*{thm*}{Theorem}
\newtheorem*{conj*}{Conjecture}

\newtheorem*{remark}{Remark}

\newtheorem{thm}{Theorem}[section]
\newtheorem{lem}{Lemma}[section]
\newtheorem{cor}[thm]{Corollary}

\newtheorem{prop}[thm]{Proposition}

\newtheorem*{example}{Example}

\newcommand{\ord}{\mathrm{ord}}
\newcommand{\Z}{\mathbb{Z}}

\newcommand{\SL}{\operatorname{SL}}

\newcommand{\leg}[2]{\genfrac{(}{)}{}{}{#1}{#2}}

\numberwithin{equation}{section}

\renewcommand*{\backref}[1]{}
\renewcommand*{\backrefalt}[4]{%
  \ifcase #1 %
    \relax
  \or
    $\uparrow$#2.%
  \else
   $\uparrow$#2.%
  \fi%
}

\begin{document}
\title[Even values of Ramanujan's tau-function]{Even values of Ramanujan's tau-function}
 \author[J. S. Balakrishnan, K. Ono, and W.-L. Tsai]{Jennifer S. Balakrishnan, Ken Ono, and Wei-Lun Tsai}
\dedicatory{In celebration of Don Zagier's 70th birthday}

 \address{Department of Mathematics and Statistics, Boston University,
Boston, MA 02215}
\email{jbala@bu.edu}

\address{Department of Mathematics, University of Virginia, Charlottesville, VA 22904}
 \email{ken.ono691@virginia.edu}
\email{wt8zj@virginia.edu}

 \thanks{The first author acknowledges the support of NSF grant (DMS-1702196 and DMS-1945452), the Clare Boothe Luce Professorship
(Henry Luce Foundation), a Simons Foundation grant (Grant \#550023), and a Sloan
 Research Fellowship. 
 The second author thanks  the Thomas Jefferson Fund and the NSF
(DMS-1601306, DMS-2002265 and DMS-2055118).}
\keywords{Lehmer's Conjecture, Ramanujan's Tau-function, Newforms, Modular forms}

\begin{abstract} 
In the spirit of Lehmer's speculation that Ramanujan's tau-function never vanishes, it is natural to ask whether any given integer $\alpha$ is a value of $\tau(n)$. For odd $\alpha$, Murty, Murty, and Shorey proved that $\tau(n)\neq \alpha$ for sufficiently large $n$. Several recent papers have identified
  explicit examples of odd $\alpha$ which are not tau-values. 
Here we apply these results (most notably the recent work of Bennett, Gherga, Patel, and Siksek) to offer the first examples of even integers that are not tau-values. Namely, for primes $\ell$ we find that
 \begin{displaymath}
 \tau(n)\not \in \{ \pm 2\ell \ : \ 3\leq \ell< 100\} \cup
\{\pm 2\ell^2 \ : \ 3\leq \ell <100\} \cup \{\pm 2\ell^3 \ : \ 3\leq \ell<100\ {\text {\rm with $\ell\neq 59$}}\}.
\end{displaymath}
Moreover, we obtain such results for infinitely many powers of each prime $3\leq \ell<100$. As an example, for $\ell=97$ we prove that
$$
\tau(n)\not \in \{ 2\cdot 97^j \ : \ 1\leq j\not \equiv 0\pmod{44}\}\cup \{-2\cdot 97^j \ : \ j\geq 1\}.
$$
The method of proof applies {\it mutatis mutandis} to newforms with residually reducible mod 2 Galois representation and is easily adapted to generic newforms with integer coefficients.
\end{abstract}
\maketitle
\section{Introduction and statement of results}

Ramanujan's tau-function \cite{RamanujanUnpublished, Ramanujan}, the coefficients of 
 the unique normalized weight 12 cusp form for $\SL_2(\Z)$
  (note: $q:=e^{2\pi i z}$ throughout)
\begin{equation}
\Delta(z)=\sum_{n=1}^{\infty}\tau(n)q^n:=q\prod_{n=1}^{\infty}(1-q^n)^{24}=q-24q^2+252q^3-1472q^4+4830q^5-\cdots,
\end{equation}
has been a remarkable prototype in the theory of modular forms.
Despite many advances that reveal its deep properties, Lehmer's Conjecture \cite{Lehmer} that $\tau(n)$ never vanishes remains open.

 In the spirit of this conjecture, it is natural to ask whether any given integer $\alpha$ is a value of $\tau(n).$
 Much is known for odd $\alpha$ thanks to the convenient fact that
 \begin{equation}\label{ThetaCong}
  \Delta(z)\equiv \sum_{n=0}^{\infty}q^{(2n+1)^2}\pmod 2.
 \end{equation}
 Murty, Murty, and Shorey \cite{MMS} proved that
 $\tau(n)\neq \alpha$ for sufficiently large $n$. 
Craig and the authors \cite{BCO, BCOT} proved some effective results concerning potential odd values of $\tau(n)$ and, more generally, coefficients of newforms with residually reducible mod $2$ Galois representation. Their methods have been carried further in subsequent work by
Amir and Hong \cite{AmirHong}, Dembner and Jain \cite{DembnerJain}, and Hanada and Madhukara \cite{HanadaMadhukara}.
For example, for $n>1$, these papers prove that
\begin{equation}\label{Known}
\tau(n) \not \in \{\pm 1, \pm  691\}\cup \{\pm \ell \ : \ 3\leq \ell <100 \ {\text {\rm prime}}\}.
\end{equation}
Recently, Bennett, Gherga, Patel, and Siksek \cite{BGPS} proved a number of spectacular results regarding odd values of $\tau(n).$
For example, they prove (see Theorem 6 of \cite{BGPS}) that $|\tau(n)|\neq \ell^b$, where
$3\leq \ell<100$ is prime and $b$ is a positive integer.

Much less is known for even $\alpha$. To this end, we make use of lower bounds for the number of prime divisors of tau-values. Craig and the authors 
proved (see\footnote{Theorem 2.5 of \cite{BCOT} concerns the case of generic newforms with integer coefficients.} Theorem 1.5 of \cite{BCOT}) that
\begin{equation}\label{NumPrimeFactors}
\Omega(\tau(n))\geq \sum_{\substack{p\mid n\\ prime}}\left ( \sigma_0(\ord_p(n)+1)-1\right) \geq \omega(n),
\end{equation}
where $\omega(n)$ (resp. $\Omega(\tau(n))$) is the number of distinct prime factors of $n$  (resp.  $\tau(n)$ with multiplicity), and $\sigma_0(N)$ is the number of positive divisors of $N$. 
Therefore, if $\tau(n)=\pm 2,$ then $n=p^m,$ where $p$ and $m+1$ are both prime\footnote{In Section 2 we shall show that $\tau(n)=\pm 2$ requires that $n$ is prime.}.
Similarly, if $\tau(n)=\pm 2\ell,$ where $\ell$ is an odd prime, then
this inequality implies that $n$ has at most two distinct prime factors. Moreover, if $n=p_1^{m_1}p_2^{m_2}$, where $p_1\neq p_2$ are prime and $m_1, m_2\geq 1,$ then
$m_1+1$ and $m_2+1$ are both prime.

Combining these results with the recent work of Bennett, Gherga, Patel, and Siksek \cite{BGPS}, we show that certain even numbers never arise as tau-values. To make this precise, we require sets of triples $(\ell,r,t),$ where $3\leq \ell <100$ is prime and $r\!\!\pmod t$ is an arithmetic progression with modulus $t\mid 44$:
\begin{equation}\label{Splus}
S^{+}:= \left\{ \begin{matrix} (3, 0, 44), (5,0,22), (7,0,44), (7,19,44), (11,0,22), (13, 0, 44), (17, 0, 44), \\ (19, 0, 22),  (23,0,4), (29,0,22), 
(31,0,22), (37,0,44), (37,35,44), (41,0,22),\\ (43,0,44), (43,37,44), (47,0,4), (53,0,44), (59, 0, 22), (61,0,22), (67, 0, 44), \\ (67,43,44),
(71,0,22), (73,0,44), (79,0,22), (83,0,44), (89,0,22), (97,0,44)
\end{matrix}
\right \}
\end{equation}
\begin{equation}\label{Sminus}
S^{-}:=\{(3,15, 44), (5, 11, 22), (17, 33, 44), (59, 3, 22), (83,11, 44), (89, 11, 22)\}.
\end{equation}
Then we define the set of pairs
\begin{equation}\label{Goodj}
N^{\pm}:=\{ (\ell, j) \ : \ 1\leq j\not \equiv r\!\!\!\pmod t \ {\text {\rm for all}}\ (\ell, r,t)\in S^{\pm}\}.
\end{equation}
These sets determine values of the form $\pm 2\cdot \ell^j$ that we rule out as possible even tau-values.

\begin{thm}\label{MainThm}
If $j\geq 1$ and $3\leq \ell<100$ is prime, then for every $n$ we have
$$
\tau(n)\not \in \{2\ell^j \ : \ (\ell,j)\in N^{+}\} \cup \{-2\ell^j\ : \ (\ell,j)\in N^{-}\}.
$$
Moreover, we have that $\tau(n)\not \in \{\pm 2\cdot 691\}.$
\end{thm}

\begin{example} The triples $(7,r,t)\in S^{+}$ are $(7, 0,44)$ and $(7, 19, 44).$ Therefore, Theorem~\ref{MainThm} gives
$$
\tau(n)\not \in \{2\cdot 7^j \ : \  j\not \equiv 0, 19\pmod{44}\}.
$$
\end{example}

\begin{example}
Let $\Omega:= \{7, 11, 13, 19, 23, 29, 31, 37, 41, 43, 47, 53, 61, 67, 71, 73, 79, 97\}$ be the set of primes
$3\leq \ell<100$ for which there are no triples of the form $(\ell,r,t)\in S^{-}$.  For these primes,
$N^{-}$ contains $(\ell,j)$ for every $j\geq 1,$ and so Theorem~\ref{MainThm} gives
$$
\tau(n)\not \in \{-2\ell^j \ : \ \ell \in \Omega \ {\text {\rm and}}\ j\geq 1\}.
$$
\end{example}

As an immediate corollary, we obtain the following conclusion for primes $3\leq \ell<100.$

\begin{cor}
For every $n,$ we have
\begin{displaymath}
\tau(n)\not \in \{ \pm 2\ell \ : \ 3\leq \ell< 100\} \cup
\{\pm 2\ell^2 \ : \ 3\leq \ell <100\} \cup \{\pm 2\ell^3 \ : \ 3\leq \ell<100\ {\text {\rm with $\ell \neq 59$}}\}.
\end{displaymath}
\end{cor}

\begin{remark} 
The first examples
of $\tau(n)=\pm 2\ell,$ where $\ell$ is prime, are
$$\tau(277)=-2\cdot 8209466002937 \ \ \
{\text {\rm and}}\ \ \
\tau(1297)=2\cdot 58734858143062873.
$$
We note that $277$ and $1297$ are both prime. Every such value with $n\leq 200,000$ has prime $n$. 
\end{remark}

The proof of Theorem~\ref{MainThm} is a modification of the method employed in  \cite{BCO, BCOT}.  These tools are based on the observation that integer sequences of the form
$\{1, \tau(p), \tau(p^2), \tau(p^3),\dots\},$
where $p$ is prime, are
 {\it Lucas sequences}.  
Important work of Bilu, Hanrot, and Voutier \cite{BHV} on primitive prime divisors of Lucas sequences  applies to $\alpha$-variants of Lehmer's Conjecture. Loosely speaking, their work implies that each $\tau(p^m)$ is divisible by at least one prime $\ell$  for which $\ell \nmid \tau(p)\tau(p^2)\cdots \tau(p^{m-1}).$ In \cite{BCO, BCOT}, this property is combined with the theory of newforms to obtain variants of Lehmer's Conjecture. Namely, certain odd integers $\alpha$ are ruled out as tau-values, as well as coefficients of newforms with residually reducible mod 2 Galois representation. 
Such conclusions follow from the absence of special integer points $(X,Y)$ on  specific curves, including hyperelliptic curves and curves defined by Thue equations. These special points (if any) have the property that $X=p$ or $p^{2k-1}$, where $p$ is prime and $2k$ is the weight of the newform.

 In Section~\ref{NutsAndBolts}, we recall the main tools from \cite{BCOT} and essential facts
about newform coefficients, such as Ramanujan's tau-function. In Section~\ref{Proof} we combine these 
facts with (\ref{Known}), the work of Bennett, Gherga, Patel, and Siksek (i.e. Theorem 6 of \cite{BGPS}),  and Ramanujan's famous tau-congruences to prove Theorem~\ref{MainThm}.

\begin{remark}
The proof of Theorem~\ref{MainThm} applies {\it mutatis mutandis} to integer weight newforms with integer coefficients and residually reducible mod 2 Galois representation. A minor modification holds
for arbitrary integer weight newforms $f(z)$ with integer coefficients, regardless of its $2$-adic properties. Indeed, suppose that $f(z)=\sum_{n=1}^{\infty}a_f(n)q^n,$ and let $\alpha$ be any non-zero integer.
We consider the ``equation''
$a_f(n)=\alpha.$
Theorem 2.5 of \cite{BCOT} offers the generalization of
(\ref{NumPrimeFactors}) which constrains the possible prime factorizations of $n$; the number of distinct prime factors of $n$ generally does not exceed $\omega(\alpha)$.
By the multiplicativity of newform coefficients, for $d\mid \alpha$, we must solve the equation $a_f(p^m)=d$,
where $m\geq 1,$ and $p$ is prime.
To this end, one applies Theorem~3.2 of \cite{BCOT} which identifies the finitely many $m$ that must be considered.
As explained in \cite{BCOT}, a solution for $p,$ when $m\geq 2$, requires special integer points
on specific curves. In many cases, there are no such points, which leads to
restrictions such as those  in Theorem~\ref{MainThm} using the methods employed in  \cite{BCO, BCOT, BGPS}.
\end{remark}

\section*{Acknowledgements} 
We thank Matthew Bisatt for several helpful discussions about root numbers of Jacobians of hyperelliptic curves which were part of initial discussions related to this project. We are indebted to the referees who made suggestions that substantially improved the results in this paper.
Finally, we thank Kaya Lakein and Anne Larsen for comments which improved the exposition in this paper.

\section{Nuts and bolts}\label{NutsAndBolts}

Here we recall essential facts about Lucas sequences and properties of
newform coefficients. 

\subsection{Properties of Newforms}

We recall basic facts about even integer weight newforms
(see \cite{AtkinLehner}), along with the deep theorem of Deligne  \cite{Deligne1, Deligne2} that bounds their Fourier coefficients. 

\begin{thm}\label{Newforms} Suppose that $f(z)=q+\sum_{n=2}^{\infty}a_f(n)q^n\in S_{2k}(\Gamma_0(N))$ is a newform with integer coefficients.
Then the following are true:
\begin{enumerate}
\item If $\gcd(n_1,n_2)=1,$ then $a_f(n_1 n_2)=a_f(n_1)a_f(n_2).$
\item If $p\nmid N$ is prime and $m\geq 2$, then
$$
a_f(p^m)=a_f(p)a_f(p^{m-1}) -p^{2k-1}a_f(p^{m-2}).
$$
\item If $p\nmid N$ is prime and $\alpha_p$ and $\beta_p$ are roots of $F_p(x):=x^2-a_f(p)x+p^{2k-1},$ then
$$
   a_f(p^m)=\frac{\alpha_p^{m+1}-\beta_p^{m+1}}{\alpha_p-\beta_p}.
$$   
Moreover, we have $|a_f(p)|\leq 2p^{\frac{2k-1}{2}}$, and $\alpha_p$ and $\beta_p$ are complex conjugates.
\end{enumerate}
\end{thm}

\subsection{Implications of properties of Lucas sequences for newforms}

Suppose that $\alpha$ and $\beta$ are algebraic integers for which $\alpha+\beta$ and $\alpha \beta$
are relatively prime non-zero integers, where $\alpha/\beta$ is not a root of unity.
Their {\it Lucas numbers} $\{u_n(\alpha,\beta)\}=\{u_1=1, u_2=\alpha+\beta,\dots\}$ are the integers
\begin{equation}
u_n(\alpha,\beta):=\frac{\alpha^n-\beta^n}{\alpha-\beta}.
\end{equation}
In particular, in the notation of Theorem~\ref{Newforms}, for primes $p\nmid N$ and $m\geq 1$, we have
\begin{equation}\label{NewformLucas}
   a_f(p^m)=u_{m+1}(\alpha_p,\beta_p)=\frac{\alpha_p^{m+1}-\beta_p^{m+1}}{\alpha_p-\beta_p}.
\end{equation}

The following well-known relative divisibility property is important for the proof of Theorem~\ref{MainThm}.

\begin{prop}[Prop. 2.1 (ii) of \cite{BHV}]\label{PropA}  If $d\mid n$, then $u_d(\alpha, \beta) | u_n(\alpha,\beta).$
\end{prop}

To prove Theorem~\ref{MainThm}, we employ bounds on the first occurrence of a multiple of a prime $\ell$ in a Lucas sequence.
 We let $m_{\ell}(\alpha,\beta)$ be the smallest $n\geq 2$
for which $\ell \mid u_n(\alpha,\beta)$. We note that $m_{\ell}(\alpha,\beta)=2$ if and only if
$\alpha +\beta\equiv 0\pmod {\ell}.$ The following proposition is well known.

\begin{prop}[Corollary\footnote{This corollary is stated for Lehmer numbers. The conclusions hold for Lucas numbers because $\ell \nmid (\alpha+\beta)$.}  2.2 of \cite{BHV}]\label{PropB} If $\ell\nmid \alpha \beta$ is an odd prime with
$m_{\ell}(\alpha,\beta)>2$, then the following are true.
\begin{enumerate}
\item If $\ell \mid (\alpha-\beta)^2$, then $m_{\ell}(\alpha,\beta)=\ell.$
\item If $\ell \nmid (\alpha-\beta)^2$, then $m_{\ell}(\alpha,\beta) \mid (\ell-1)$ or $m_{\ell}(\alpha,\beta)\mid (\ell+1).$
\end{enumerate}
\end{prop}

\begin{remark}
If $\ell \mid \alpha \beta$, then either $\ell \mid u_n(\alpha,\beta)$ for all $n$, 
or $\ell \nmid u_n(\alpha,\beta)$ for all $n$.
\end{remark}

A prime  $\ell \mid u_{n}(\alpha,\beta)$ is a {\it primitive prime divisor of $u_n(\alpha,\beta)$} if $\ell \nmid (\alpha-\beta)^2 u_1(\alpha,\beta)\cdots u_{n-1}(\alpha, \beta)$.
Bilu, Hanrot, and Voutier \cite{BHV} proved that 
every Lucas number $u_n(\alpha,\beta)$, with $n>30,$
has a primitive prime divisor.
Their work is comprehensive;
they have classified {\it defective} terms, the integers
 $u_n(\alpha,\beta)$, with $n>2,$ that do 
not have a primitive prime divisor. 
 Their work, combined with a subsequent paper\footnote{This paper included a few cases that were omitted in \cite{BHV}.} 
by Abouzaid \cite{Abouzaid},   gives the {\it complete classification} of
defective Lucas numbers.
In \cite{BCO, BCOT}, these results were applied to even weight newforms, including $\Delta(z).$
Arguing as in these papers, we obtain the following lemma.

\begin{lem}\label{PPD}
Suppose $2k\geq 4$ is even, and  $\alpha$ and $\beta$ are roots of the integral polynomial
\begin{align}\label{Qpoly}
F(X)=X^2-AX+p^{2k-1}=(X-\alpha)(X-\beta),
\end{align}
where $p$ is prime, $|A|=|\alpha+\beta|\leq 2p^{\frac{2k-1}{2}},$ and $\gcd(\alpha+\beta, p)=1.$ Then there are no defective Lucas numbers {\color{black}$u_n(\alpha,\beta)\in \{\pm 2\ell^i\},$ where $i\geq0$ and} $\ell$ is an odd prime. Also, if $u_n(\alpha,\beta)=\pm\ell$ is a defective Lucas number, then one of the following is true.

\begin{enumerate}
\item We have $(A,\ell, n) = (\pm m, 3, 3),$ where $3\nmid m$ and $(p, \pm m)$ satisfies $Y^2 = X^{2k-1} \pm 3.$

\item We have $(A, \ell, n) =
(\pm \ell, \ell, 4),$ where $(p, \pm \ell)$ satisfies $Y^2 = 2X^{2k-1}-1$.
\end{enumerate}
\end{lem}

\begin{remark}
Thanks to Lemma~\ref{PPD} and (\ref{NumPrimeFactors}), if $\tau(n)=\pm 2,$ then $n$ must be prime.
\end{remark}
\begin{proof}

As mentioned above,  \cite{Abouzaid, BHV} classify defective Lucas numbers. This classification includes a finite list of sporadic examples and a  list of parameterized infinite families. Theorem 2.2 of \cite{BCOT}  uses these results to describe the defective
Lucas numbers that can arise as newform coefficients, i.e. 
sequences defined by (\ref{Qpoly}). Tables 1 and 2 of \cite{BCOT} list the possible defective cases.

An inspection of Table 1 of \cite{BCOT}, which concerns the sporadic examples, reveals that the only possible defective numbers with $2k\geq 4$ have $2k=4.$  Moreover, they are the odd numbers $u_3(\alpha,\beta)=1$ or $u_4(\alpha,\beta)=\pm85.$

To complete the proof, we  consider the parametrized infinite families in Table 2 of \cite{BCOT}.
If  $u_n(\alpha,\beta)$ is even, then we only have to consider rows four, five, six, and seven of the table.  A simple inspection reveals that {\color{black} $\{\pm 2\ell^i\}$ for $i\geq0$} never arises.
This then leaves $u_n(\alpha,\beta)=\pm\ell$ as the only cases to consider.
However, Lemma 2.1 of \cite{BCOT} includes these cases, giving  (1) and (2) above.
\end{proof}

\section{Proof of Theorem ~\ref{MainThm}}\label{Proof}
Here we use
 the previous section to prove Theorem~\ref{MainThm}.

\subsection{Ramanujan's Congruences}

Ramanujan's classical congruences for the tau-function imply the following convenient fact involving the sets
$N^{\varepsilon}$ defined in (\ref{Goodj}).

\begin{lem}\label{CongruenceReduction} 
If $3\leq \ell<100$ is prime and
$(\ell,j)\in N^{\varepsilon},$ then for every prime $p$ we have that
$$\tau(p)\neq \varepsilon 2 \ell^j.$$
\end{lem}
\begin{proof}
We recall the famous
Ramanujan congruences 
(see \cite{RamanujanUnpublished, Ramanujan}):
\begin{displaymath}
\tau(n)\equiv \begin{cases} n^3\sigma_1(n)\pmod 4,\\
n^2\sigma_1(n)\pmod 3,\\
n\sigma_1(n)\pmod 5,\\
n\sigma_3(n)\pmod 7.
\end{cases}
\end{displaymath}
where $\sigma_v(n):=\sum_{1\leq d\mid n}d^v.$
Furthermore, if $p\neq 23$ is prime, Ramanujan proved that
$$
\tau(p)\equiv \begin{cases} 0\pmod{23} \ \ \ \ &{\text {\rm if }} \leg{p}{23}=-1,\\
\sigma_{11}(p)\pmod{23^2} \ \ \ \ &{\text {\rm if  $p=a^2+23b^2$ with $a,b\in \Z$}},\\
-1\pmod{23} \ \ \ \ &{\text {\rm otherwise.}}
\end{cases}
$$
If $p\neq 23$ is prime, then the collection of these congruences imply
\begin{displaymath}
\begin{split}
&\tau(p)\equiv 0\!\!\!\pmod 2,\  \tau(p)\equiv 0, 2\!\!\!\pmod 3,\  \tau(p)\equiv 0, 1, 2\!\!\!\pmod 5,\\
&\tau(p)\equiv 0, 1, 2, 4\!\!\!\pmod 7, \ \ {\text {\rm and}} \ \ \tau(p)\equiv 0, -1, 2\!\!\!\pmod{23}.
\end{split}
\end{displaymath}
These congruences are easily reformulated in terms of
$N^{\varepsilon}.$
This completes the proof for $p\neq 23.$
Finally, we note that
$\tau(23)=18643272= 2^3\cdot 3\cdot 617\cdot 1259.$
\end{proof}

\subsection{Proof of Theorem~\ref{MainThm}}

Theorem~\ref{MainThm} consists of two different types of $\alpha.$ 
\begin{enumerate}
\item The case where $\alpha=\pm 2\ell,$ where $3\leq \ell \leq 100$ is prime or $\ell=691.$
\item The case where $\alpha=\pm 2\ell^j,$ where $3\leq \ell \leq 100$ is prime and $j\geq 2.$
\end{enumerate}
By Lemma~\ref{PPD} with $2k=12,$ the numbers {\color{black}$\{\pm 2\ell^i\}$ for $i\geq0$}
(if they arise) are never defective Lucas numbers in 
$\{\tau(p), \tau(p^2), \tau(p^3),\dots\},$
where $p$ is prime.  Lemma~\ref{PPD} (1) and (2) covers the cases apart from $\pm \ell,$ which were ruled out by Lemma 2.1 of \cite{BCO}.

\smallskip
\noindent
{\bf Case} (1). Thanks to (\ref{NumPrimeFactors}), if $\tau(n)=\pm 2\ell,$ where $\ell$ is an odd prime, then
either $n=p_1^{m_1}$, or $n=p_1^{m_1}p_2^{m_2}$, where the $p_i$ are prime and the $m_i\geq 1.$
Using Theorem~\ref{Newforms} (1) and (\ref{Known}), the latter case requires
$|\tau(p_1^{m_1})|=2$ and $|\tau(p_2^{m_2})|=\ell.$
Thanks to (\ref{Known}) again, this is impossible for $\ell=691$ and primes $3\leq \ell < 100.$

Therefore, we may assume that
$\tau(p_1^{m_1})=\pm 2\ell.$  Thanks to Theorem~\ref{Newforms}, we have that $p_1\neq 2,$ as $4\mid \tau(2^m)$ for every positive integer $m.$
Therefore,
(\ref{ThetaCong}) implies that $m_1$ is odd.
Moreover, since $\tau(p_1)$ is even, it must be that $\tau(p_1^{m_1})$ is the first term in the Lucas sequence that is divisible by $\ell.$ Otherwise, $\pm 2\ell$ would be defective, contradicting Lemma~\ref{PPD}. {\color{black} If $m_1+1$ has a non-trivial divisor other than $2$, then by the relative divisibility of Lucas numbers given in Proposition~\ref{PropA}, and the nondefectivity of $\pm 2$ in Lemma~\ref{PPD}, we obtain a contradiction. Hence, it boils down to considering the case when $m_1+1=4$, $\tau(p)=\pm 2$, and $\tau(p^3)=\pm 2\ell$ for some prime $p$. However, using the Hecke relation ((2) in Theorem~\ref{Newforms}), we have that  $\pm2\ell=\pm4(p^{11}-2)$, and there is no such $p,$ as the left hand side is $2\!\pmod 4$ while the right hand side is $0\!\pmod 4.$ Hence, it follows that $m_1+1$ is prime.}

Therefore, we have $m_1=1,$ which in turn leads to  $\tau(p_1)=\pm 2\ell.$ The proof in this case is complete as
Lemma~\ref{CongruenceReduction}  shows that $\tau(p)\neq \pm 2 \ell.$

\smallskip
\noindent
{\bf Case} (2). Since $3\leq \ell < 100$ is prime, (\ref{Known}) and Theorem 6 of \cite{BGPS} implies that $|\tau(n)|\neq \ell^b$ for all $n$ and $b\geq 1.$
Therefore, we may assume that $\tau(p^{m})=\pm 2\ell^j,$ where $p$ is an odd prime.  Here we again use the fact that $4\mid \tau(2^m)$ for every positive integer $m,$ {\color{black}and consider the degenerate case $m_1+1=4$, $\tau(p)=\pm 2$, and $\tau(p^3)=\pm 2\ell^j$ for some prime $p$, which gives the corresponding equation $\pm2\ell^j=\pm4(p^{11}-2).$} The argument  in Case~(1), where the conclusion is that $m=1$, applies {\it mutatis mutandis}. Therefore, the proof is complete as Lemma~\ref{CongruenceReduction} shows that $\tau(p)\neq \pm 2\ell^j.$

\end{document}